\documentclass[reqno]{amsart}
\usepackage{amsmath,amsfonts}
\usepackage{color}
\usepackage{hyperref}
\usepackage[inner=1.0in,outer=1.0in,bottom=1.0in, top=1.0in]{geometry}

\newtheorem{theorem}{Theorem}[section]
\newtheorem{lemma}[theorem]{Lemma}
\newtheorem{proposition}[theorem]{Proposition}
\newtheorem{corollary}[theorem]{Corollary}
\theoremstyle{definition}

\theoremstyle{remark}
\newtheorem{remark}[theorem]{Remark}

\numberwithin{equation}{section}

\begin{document}
\author{Junehyuk Jung}
\address{Department of Mathematical Science, KAIST, Daejeon 305-701, South Korea}
\email{junehyuk@kaist.ac.kr}
 \author{Matthew P. Young}
 \address{Department of Mathematics \\
 	  Texas A\&M University \\
 	  College Station \\
 	  TX 77843-3368 \\
 		U.S.A.}
 \email{myoung@math.tamu.edu}
 \thanks{This material is based upon work supported by the National Science Foundation under agreement No. DMS-1128155. (J.J. and M.Y.) and No. DMS-1401008 (M.Y.) .  Any opinions, findings and conclusions or recommendations expressed in this material are those of the authors and do not necessarily reflect the views of the National Science Foundation.   The first author was also partially supported by the National Research Foundation of Korea (NRF) grant funded by the Korea government(MSIP)(No. 2013042157).
}

 \begin{abstract}
We prove a quantitative lower bound on the number of nodal domains of the real-analytic Eisenstein series.  The main tool in the proof is a quantitative restricted QUE theorem where the support of the test function is allowed to shrink with the Laplace eigenvalue.
 \end{abstract}
 \title{Sign changes of the Eisenstein series on the critical line}
 \maketitle
\section{Introduction and statement of results}
Let $(M,g)$ be a smooth Riemannian manifold and let $\phi$ be a real-valued Laplace eigenfunction with the eigenvalue $\lambda^2$. The asymptotic behavior of the nodal set $Z_\phi = \{p\in M~:~ \phi(p)=0\}$ of $\phi$, as $\lambda^2 \to +\infty$, is fundamental in spectral geometry. Nodal domains are the connected components of $M \backslash Z_\phi$, and we denote by $N(\phi)$ the number of nodal domains. The most basic question regarding nodal domains one may ask is if $N(\phi)$ tends to infinity as the eigenvalue grows. Note that there are examples \cite{st,lewy,crit} where there are a bounded number of nodal domains for an infinite sequence of eigenfunctions, so this question is sensitive to the manifold.


Recently, Ghosh, Reznikov, and Sarnak \cite{GRS} have studied this question for the sequence of Hecke-Maass cusp forms on the modular surface $\mathbb{X}=SL_2(\mathbb{Z}) \backslash \mathbb{H}$, and shown that the number of nodal domains grows with the eigenvalue (with a quantitative growth rate) on the assumption of the (unproved) generalized Lindel\"{o}f Hypothesis.  Jang and Jung \cite{JangJung} have unconditionally shown that the number of nodal domains goes to infinity with the eigenvalue, however with no rate of growth.

In this article, we study the nodal domains for the Eisenstein series $E_T(z) = E(z,1/2+iT)$ on $SL_2(\mathbb{Z})\backslash \mathbb{H}$.
Let $E_T^*(z) = \frac{\theta(1/2 +
iT)}{|\theta(1/2 + iT)|} E_T(z)$, where $\theta(s) = \pi^{-s} \Gamma(s)
\zeta(2s)$.  This normalization makes $E_T^*$ real-valued for $z \in \mathbb{H}$.
Our main result is
\begin{theorem}\label{thm:signchange}
Fix $0 \leq \nu < 1/51$.
Then $E_T^*(iy)$ has $\gg T^{\nu}$ sign changes along $[1,3]$, for all sufficiently large $T$.
\end{theorem}

Note that if $(M,g)$ has an isometric involution $\tau : M \to M$ such that the fixed point set $\mathrm{Fix}(\tau)$ contains a geodesic segment $\gamma$, then the sign changes along $\gamma$ give rise to nodal domains \cite{GRS, JZ1, GRS2}. Since $x+iy \mapsto -x+iy$ on $\mathbb{H}$ induces an orientation reversing isometric involution on $\mathbb{X}$, whose fixed point set contains the geodesic segment $\{iy\in \mathbb{H}~:~1\leq y\leq 3\}$, we infer from Theorem \ref{thm:signchange} that:


\begin{corollary}
For $0 \leq \nu < 1/51$, $E_T^*(z)$ has $\gg T^{\nu}$ nodal domains that meet the segment $\{iy\in \mathbb{H}~:~ 1 \leq y \leq 3\}$.
\end{corollary}
\begin{remark}
For very large values of $y$ (say, $y > T$), the Eisenstein series is closely approximated by the constant term in its Fourier expansion, and it is easy to show that there are infinitely many ``trivial" nodal domains of $E_T^*(z)$ for $y >T$.

The method of Ghosh, Reznikov, and Sarnak \cite{GRS} produces sign changes of Hecke-Maass cusp forms high in the cusp, $ T>y>\frac{T}{100}$.

\end{remark}

\section{Quantum Unique Ergodicity and sign-detecting symbols}
In this section, we describe the motivation behind the proof. A similar idea is used in \cite{JangJung}, where the authors prove that the number of nodal domains of certain eigenfunctions tends to infinity.

\subsection{Background on Quantum Ergodicity}
Let $(M,g)$ be a compact smooth Riemanian manifold without boundary, and let $\{u_j\}_{j=1,2,\ldots}$ be an orthonormal eigenbasis of the Laplace--Beltrami operator $\Delta_g$ on $M$. We denote by $\lambda_j^2$ the eigenvalue corresponding to $u_j$, and we assume that $0=\lambda_1<\lambda_2 \leq \lambda_3 \leq \cdots$. For a symbol $a \in C^\infty (T^*M)$, we denote by $\mathrm{Op}(a)$ the quantum observable obtained through the standard quantization of $a$. Here $T^*M$ is the cotangent bundle of $M$. The Quantum Ergodicity theorem by Colin de Verdi\'ere, Shnirelman, and Zelditch \cite{CdV,Snirelman,Zelditch} implies that if the geodesic flow defined on $M$ is ergodic, then there exists a density one subset $B$ of $\mathbb{N}$ such that the following is true:
\[
\lim_{\substack{j \to +\infty\\j \in B}} \langle \mathrm{Op}(a) u_j, u_j \rangle_M = \int_{S^*M} a \thinspace d\mu,
\]
where $d\mu$ is the Liouville measure on the unit cotangent bundle $S^*M$. We say Quantum Unique Ergodicity (QUE) holds for $\{u_j\}_{j=1,2,\ldots}$, when we can take $B=\mathbb{N}$.

Now assume that there exists an orientation reversing isometric involution $\tau: M \to M$ such that the fixed point set $\mathrm{Fix}(\tau)$ contains a hypersurface $H \subset M$. Assume further that $\{u_j\}_{j=1,2,\ldots}$ is an eigenbasis for the space of even $L^2$ functions. In this case, the Quantum Ergodic Restriction theorem \cite{ctz,dz,tz1} implies that there exists a density one subset $B$ of $\mathbb{N}$ such that for any compactly supported smooth function $f \in C_0^\infty (H)$, we have
\[
\lim_{\substack{j \to +\infty\\j \in B}}\langle f u_j, u_j \rangle_H = 2\int_H f d\mu_H.
\]
If QUE for the restriction to $H$ is true, then one can take $B=\mathbb{N}$.

In the following section, we are going to discuss how one can detect sign changes of eigenfunctions along a curve using these two ingredients.

\subsection{Detecting sign changes}
Fix a point $x_0$ on $H$ and consider Fermi normal coordinates $x=(x_1, \ldots, x_n)$ of $M$ near $x_0$. In these coordinates, $(x_1,\ldots, x_{n-1},0)=(x',0)$ represent points on $H$. Let $\xi$ and $\xi'$ be covectors corresponding to $x$ and $x'$ respectively. We denote by $B^*H$ a unit ball bundle over $H$, which is obtained by restricting base points of $S^*M$ to $H$, i.e.,
\[
B^*H := \{(x,\xi)\in S^*M~:~ x \in H\}.
\]
In a recent paper  by Christianson, Toth, and Zelditch \cite{ctz}, the authors prove that QUE for $\{u_j\}$ on the ambient manifold $M$ implies a certain unique behavior of the restrictions $\{u_j|_H\}$ to the hypersurface. This in particular implies that, assuming QUE, we have
\begin{equation}\label{eq1}
\langle \mathrm{Op}_H(a) (1+\lambda_j^{-2}\Delta_H)u_j, u_j \rangle_H \xrightarrow{j\to +\infty} \frac{4}{\mu (S^*M)}  \int_{B^*H} a(x,\xi)(1-|\xi'|^2)^{1/2}d\sigma,
\end{equation}
where $\{u_j\}_{j=1,2,\ldots}$ is an eigenbasis for the space of even $L^2$ functions. Here $d\sigma$ is the Liouville measure on $B^*H$. If we assume further that QUE for the restriction to $H$ is true, then one can in fact remove $(1+\lambda_j^{-2}\Delta_H)$ in \eqref{eq1} and the corresponding $1-|\xi'|^2$ factor on the right hand side, so that
\begin{equation}\label{eq2}
\langle \mathrm{Op}_H(a)u_j, u_j \rangle_H \xrightarrow{j\to +\infty} \frac{4}{\mu (S^*M)} \int_{B^*H} a(x,\xi)(1-|\xi'|^2)^{-1/2}d\sigma.
\end{equation}
When $M$ is a surface and $H=\gamma$ is a geodesic segment, we may simplify \eqref{eq2} for $a(x,\xi)$ given by $f(x)\theta(\xi)$ as follows:
\begin{equation}\label{eq3}
\lim_{j \to +\infty}\langle \mathrm{Op}_\gamma(a)u_j, u_j \rangle_\gamma = \frac{2}{\pi} \int_\gamma f(x) dx \int_{-1}^1 \theta(\xi) \frac{d\xi}{\sqrt{1-\xi^2}}.
\end{equation}
Now we take $\theta(\xi)=e^{i\alpha\xi}$, so that \eqref{eq3} becomes
\begin{equation}\label{eq4}
\lim_{j \to +\infty} \int_\gamma f(x)u_j(x+\alpha\lambda_j^{-1})\overline{u_j(x)} dx = 2 J_0(\alpha) \int_\gamma f(x) dx.
\end{equation}
\begin{remark}\label{rem:arc}
Since we are taking Fermi normal coordinates, $x+\alpha\lambda_j^{-1}$ is the unique point on $\gamma$ that is $\alpha\lambda_j^{-1}$ away from $x$ in the given direction.
\end{remark}
Observe that if we take $f$ to be a non-negative function, and $\alpha$ to be any constant such that $J_0(\alpha)<0$, then this implies that all but finitely many $u_j$ have at least one sign change on a fixed open set that contains the support of $f$.

For the Eisenstein series $E\left(\cdot, \frac{1}{2}+iT\right)$ on the full modular surface $SL(2,\mathbb{Z}) \backslash \mathbb{H}$, QUE is known due to Luo--Sarnak \cite{LuoSarnakQUE} (without a rate of convergence).
QUE for Eisenstein series on $SL_2(\mathbb{Z})\backslash \mathbb{H}$ with a rate (and varying test function) was first obtained in \cite{Young}.
Earlier, Jakobson \cite{Jakobson} proved QUE for Eisenstein series on $SL_2(\mathbb{Z}) \backslash SL_2(\mathbb{R})$, but without a rate.
The restricted QUE for Eisenstein series with a rate was shown in \cite{Young}, but there the test function was held fixed.

Based on the discussion in this section, one would expect
to be able to give a quantitative lower bound for the number of sign changes of the Eisenstein series on a fixed compact geodesic segment on $\{iy~:~y>0\}$, assuming we had quantitative QUE theorems on both $SL_2(\mathbb{Z}) \backslash SL_2(\mathbb{R})$ and on the geodesic segment, where in both cases the test function must be allowed to vary.

In this article, we sidestep some of these technical problems and \emph{directly} study the left hand side of \eqref{eq4}, which is a somewhat modified version of quantitative QUE for the geodesic segment.  In this way, we completely bypass the need for quantitative QUE on $SL_2(\mathbb{Z}) \backslash SL_2(\mathbb{R})$ (which of course would be interesting for its own sake).

\section{Statement of result}
Suppose that $\psi$ is a nonnegative, smooth function with support on $[1, 3]$ (any fixed interval would be acceptable). Define
\begin{equation*}
 I_{\psi, \alpha}(T) = \int_0^{\infty} \psi(y) \psi\left(\left(1+\frac{\alpha}{T}\right)y\right)
E_T^*(iy) E_T^*\left(i\left(1+\frac{\alpha}{T}\right)y\right) \frac{dy}{y},
\end{equation*}
where $\alpha \in \mathbb{R}$ is fixed.
\begin{remark}
Note that $\mathbb{H}$ is endowed with the line element $ds^2 = y^{-2}(dx^2+dy^2)$. Hence $i\left(1+\frac{\alpha}{T}\right)y$ is the point on $\{iy~:~y>0\} \subset \mathbb{H}$ which is $\log \left( 1+\frac{\alpha}{T}\right) \sim \frac{\alpha}{T}$ away from $y$. Compare to Remark \ref{rem:arc}.
\end{remark}

In \cite[Theorem 1.2]{Young} it is shown that
\begin{equation*}
I_{\psi,0}(T) = 2 \frac{3}{\pi} \log(1/4 + T^2) \int_0^{\infty} \psi(y)^2 \frac{dy}{y} + o(\log T),
\end{equation*}
for $\psi$ fixed.
In fact, one has a fully developed main term with a power saving error term. Here we wish to allow $\alpha$ and $\psi$ to vary, which in principle should be possible due to the power-saving error term, but this requires a different proof with some new innovations.
The main technical difficulty is in bounding a certain off-diagonal main term, which occurs in Section \ref{section:ODMT} below.
The conditions we place on $\psi$ are that
\begin{equation}
 \label{eq:psibound}
 \psi^{(j)}(y) \ll_j A^j,
\end{equation}
where $A \geq 1$ is allowed to depend on $T$.  For simplicity, we shall assume that all $\psi$ have support on the same interval $[1, 3]$, but it would also be interesting to have moving support.
In the end, we will take $A$ to be a small power of $T$.  We also desire that $\psi$ is not too small everywhere.  To capture this, we assume that
\begin{equation}
\label{eq:psiintegrallowerbound}
 \int_0^{\infty} \psi(y)^2 \frac{dy}{y} \gg A^{-1}.
\end{equation}
For instance, if there exists a point $y_0 > 0$ so that $\psi^2(y_0) \geq 1$, then we can deduce \eqref{eq:psiintegrallowerbound} using the fact that $\psi(y)^2 \geq 1/2$ in a neighborhood around $y_0$ of length $\ll A^{-1}$ (using \eqref{eq:psibound} and the mean value theorem).  In addition, we assume that
\begin{equation}
\label{eq:psiprimeandpsisquaredrelation}
 A^{-2} \int_0^{\infty} \psi'(y)^2 \frac{dy}{y} \ll \int_0^{\infty} \psi(y)^2 \frac{dy}{y}.
\end{equation}
We often use the notation
\begin{equation*}
 \| \psi^2 \|_1 = \int_0^{\infty} \psi(y)^2 \frac{dy}{y}.
\end{equation*}

\begin{theorem}
\label{thm:mainthm}
Suppose that $\psi$ satisfies the above conditions with $A \ll T^{\delta}$ with $0 \leq \delta < \frac{1}{51}$.
Then for fixed $\alpha \in \mathbb{R}$, we have
\begin{equation*}
I_{\psi,\alpha}(T)
=
 2 \frac{3}{\pi} \log(1/4 + T^2) J_0(\alpha)\int_0^{\infty} \psi(y)^2 \frac{dy}{y}
 + O_{\alpha, \delta}((\log T)^{2/3+\varepsilon} \| \psi^2 \|_1).
\end{equation*}
\end{theorem}
The error term in Theorem \ref{thm:mainthm} depends continuously on $\alpha$.

\section{Notation and basic lemmas}
\subsection{Notation}
The Fourier expansion of $E_T^*$ is
\begin{equation}
\label{eq:ETFourier}
 E_T^*(x+iy) = \mu y^{1/2 + iT} + \overline{\mu} y^{1/2-iT} + \rho^*(1) \sum_{n
\neq 0} \frac{\tau_{iT}(n) e(nx)}{|n|^{1/2}} V_T(2 \pi |n| y),
\end{equation}
where $\rho^*(1) = (2/\pi)^{1/2} |\theta(1/2 + iT)|^{-1}$,
$\theta(s) = \pi^{-s} \Gamma(s) \zeta(2s)$,
$\mu =
\frac{\theta(1/2 + iT)}{|\theta(1/2 + iT)|}$, $V_T(y) = \sqrt{y} K_{iT}(y)$, $e(x)=e^{2\pi i x}$, and $\tau_{iT}(n) = \sum_{ab = |n|} (a/b)^{iT}$.
Here $\zeta(s)$ is the Riemann zeta function, and $K_{iT}(y)$ is the modified Bessel function of the second kind ($K$-Bessel function). By a well-known formula for the Mellin trasform of the $K$-Bessel function, we
have
\begin{equation}
\label{eq:gammaVformula}
 \gamma_{V_T}(1/2 + s) := \int_0^{\infty} V_T(2 \pi y) y^s \frac{dy}{y} =
2^{-3/2} \pi^{-s} \Gamma\left(\frac{1/2 + s + iT}{2}\right) \Gamma\left(\frac{1/2 +
s - iT}{2}\right).
\end{equation}
It is also helpful to recall
\begin{equation*}
 \gamma_{V_T^2}(1+s):= \int_0^{\infty} V_T(2\pi y)^2 y^s \frac{dy}{y} = 2^{-2}
\pi^{-s} \frac{\Gamma(\frac{1+s+2iT}{2}) \Gamma(\frac{1+s}{2})^2
\Gamma(\frac{1+s-2iT}{2})}{\Gamma(1+s)},
\end{equation*}
and
\begin{equation}
\label{eq:Zformula}
 Z(s,E_T) := \sum_{n=1}^{\infty} \frac{\tau_{iT}(n)^2}{n^s} = \frac{\zeta(s-2iT)
\zeta(s+2iT) \zeta(s)^2}{\zeta(2s)}.
\end{equation}

\subsection{Miscellaneous lemmas}
Here we collect some basic tools used throughout this paper.

\begin{lemma}[Vinogradov-Korobov]
 For some $c > 0$ and for any $|t| \gg 1$, $1- \frac{c}{(\log |t|)^{2/3}} \leq \sigma \leq 1$, we have
 \begin{equation*}
  \zeta(\sigma + it) \ll (\log |t|)^{2/3},
 \end{equation*}
and
\begin{equation}
\label{eq:VinogradovKorobovZetalogarithmicderivative}
 \frac{\zeta'}{\zeta}(1 + it) \ll (\log |t|)^{2/3 + \varepsilon},
\end{equation}
\end{lemma}
For a reference, see \cite[Corollary 8.28, Theorem 8.29]{IK}.

We will need Iwaniec's bound on the fourth moment of the Riemann zeta function over a short interval:
\begin{proposition}[Iwaniec \cite{IwaniecFourthMoment}]
\label{prop:Iwaniec}
Let $1 \ll U \ll V$.  Then
 \begin{equation*}
\int_{V \leq |u| \leq V + U} |\zeta(1/2 +iu)|^4 du
\ll (U + V^{2/3})^{1+\varepsilon}.
\end{equation*}
\end{proposition}


\subsection{Mellin transform of \texorpdfstring{$\psi$}{psi}}
Many of our estimates are naturally given in terms of the Mellin transform of $\psi$.  We gather here some simple estimates on $\widetilde{\psi}$.  We require some control on $\widetilde{\psi}(s)$ as $|s|\to \infty$ while $s$ is in the strip $-3<\mathrm{Re}(s) <3$. 
By integration by parts, we have
\begin{equation}
\label{eq:psiMellinIntegrationByPartsFormula}
 \widetilde{\psi}(s) = \int_0^{\infty} \psi(y) y^s \frac{dy}{y} = \frac{(-1)^j}{s(s+1)\dots(s+j-1)} \int_0^{\infty} \psi^{(j)}(y) y^{s+j} \frac{dy}{y}.
\end{equation}
Using \eqref{eq:psibound}, and recalling that $\psi$ is supported on $[1,3]$, we have
\begin{equation*}
 \widetilde{\psi}(s) \ll \frac{A^j }{|s(s+1)\dots(s+j-1)|}.
\end{equation*}
Here the implied constant depends only on $j$ and the implied constants appearing in \eqref{eq:psibound}.
 Taking $j$ large if $|\text{Im}(s)| \geq A$, and $j=0$ otherwise, we derive
\begin{equation}
\label{eq:psiMellinbound}
 \widetilde{\psi}(\sigma + it) \ll_{C} \left(1 + \frac{|t|}{A}\right)^{-C},
\end{equation}
where $C > 0$ may be chosen arbitrarily large, and the result is uniform in $-3<\sigma<3$.

It will also be convenient to mention that
\begin{equation}
\label{eq:psiMellinSquaredIntegralBound}
\frac{1}{2\pi} \int_{-\infty}^{\infty} |\widetilde{\psi}(\sigma + it)|^2 dt = \int_0^{\infty} y^{2\sigma} \psi(y)^2 \frac{dy}{y} \ll \int_0^{\infty} \psi(y)^2 \frac{dy}{y}.
\end{equation}
By a similar calculation, using \eqref{eq:psiMellinIntegrationByPartsFormula} with $j=1$, and Parseval, we have
\begin{equation}
\label{eq:psiprimeMellinSquaredIntegralBound}
 \frac{1}{2 \pi} \int_{-\infty}^{\infty} |(\sigma+it) \widetilde{\psi}(\sigma+it)|^2 dt = \int_0^{\infty} (y \psi'(y))^2 y^{2 \sigma} \frac{dy}{y} \ll \int_0^{\infty} \psi'(y)^2 \frac{dy}{y}.
\end{equation}
We may derive the following estimate:
\begin{equation}
\label{eq:psiMellinL1norm}
 \int_{-\infty}^{\infty} |\widetilde{\psi}(\sigma-iu)| du \ll A^{1/2}  \| \psi \|_2.
\end{equation}
\begin{proof}[Proof of \eqref{eq:psiMellinL1norm}]
By Cauchy-Schwarz and \eqref{eq:psiMellinSquaredIntegralBound}, we derive
\begin{equation}
\label{eq:psimellinL1boundSmalluRange}
 \int_{|u| \leq A} |\widetilde{\psi}(\sigma-iu)| du \ll A^{1/2} \left(\int_0^{\infty} y^{2 \sigma} \psi(y)^2 \frac{dy}{y} \right)^{1/2} \ll A^{1/2} \| \psi\|_2.
\end{equation}
For the part of the integral with $|u| \geq A$ we multiply and divide by $|u|/A$, apply Cauchy-Schwarz again, and use \eqref{eq:psiprimeMellinSquaredIntegralBound} as follows:
\begin{gather*}
 \int_{|u| \geq A} |\widetilde{\psi}(\sigma-iu)| du
 = \int_{|u| \geq A} \frac{|u|}{A}  |\widetilde{\psi}(\sigma-iu)| \frac{A}{|u|} du
 \ll A^{1/2} \left(\int_{|u| \geq A} \frac{|u|^2}{A^2} |\widetilde{\psi}(\sigma-iu)|^2 du \right)^{1/2}
 \\
 \leq A^{1/2} \left(\int_{-\infty}^{\infty} \frac{|\sigma+iu|^2}{A^2} |\widetilde{\psi}(\sigma+iu)|^2 du \right)^{1/2} \ll A^{1/2}  \left(A^{-2} \int_0^{\infty} \psi'(y)^2 \frac{dy}{y} \right)^{1/2}.
\end{gather*}
Using \eqref{eq:psiprimeandpsisquaredrelation} completes the proof.
\end{proof}

The main technical tool we need in this work is an estimate for a shifted divisor sum, which we directly quote as follows.
\begin{proposition}[\cite{Young}, Theorem 7.1]
\label{prop:SCS}
Suppose that $w(x)$ is a smooth function on the positive reals supported on $Y \leq x \leq  2Y$ and satisfying $w^{(j)}(x) \ll_j (P/Y)^j$ for some parameters $1 \leq P \leq Y$.  Let $\theta = 7/64$, and set $R = P + \frac{T|m|}{Y}$.  Then for $m \neq 0$, $ R \ll T/(TY)^{\delta}$, we have
\begin{equation*}
 \sum_{n \in \mathbb{Z}} \tau_{iT}(n) \tau_{iT}(n+m) w(n) = M.T. +  E.T.,
\end{equation*}
where
\begin{equation*}
 M.T. = \sum_{\pm}  \frac{|\zeta(1 + 2iT)|^2}{\zeta(2)} \sigma_{-1}(m) \int_{\max(0, -m)}^{\infty} (x + m)^{\mp iT} x^{\pm iT} w(x) dx,
\end{equation*}
and
\begin{equation*}
 E.T.\ll 
(|m|^{\theta} T^{\frac13} Y^{\frac12} R^2 
+
T^{\frac16} Y^{\frac34}  R^{\frac12} ) (TY)^{\varepsilon} .
\end{equation*}
Furthermore, with $R = P + \frac{TM}{Y}$, we have
\begin{equation*}
 \sum_{1 \leq |m| \leq M} |E.T.| \ll (M T^{\frac13} Y^{\frac12} R^2 
+
M T^{\frac16} Y^{\frac34}  R^{\frac12} ) (TY)^{\varepsilon}.
\end{equation*}
\end{proposition}

\section{Initial developments}
We have by Parseval that
\begin{equation}
\label{eq:Ialphadef}
 I_{\psi,\alpha}(T) = \frac{1}{2\pi} \int_{-\infty}^{\infty} F(-it) \left(\int_0^{\infty} E_T^*\left(i\left(1+\frac{\alpha}{T}\right)y\right) \psi\left(\left(1+\frac{\alpha}{T}\right)y\right) y^{it} \frac{dy}{y} \right) dt,
\end{equation}
where
\begin{equation}
\label{eq:Fdef}
 F(s) = \int_0^{\infty} \psi(y) E_T^*(iy) y^{s} \frac{dy}{y}.
\end{equation}
In writing this, we have assumed that $T$ is large.

We will calculate $F(s)$ here, and then perform some easy approximations.  Inserting \eqref{eq:ETFourier} into \eqref{eq:Fdef}, using \eqref{eq:gammaVformula} and the Mellin inversion formula for $\psi(y) = \frac{1}{2 \pi i} \int_{(1)} \widetilde{\psi}(-u) y^{u} du$, we obtain
\begin{multline}
\label{eq:Fformula1}
 F(s) = \mu \widetilde{\psi}(1/2 + s + iT) + \overline{\mu}
\widetilde{\psi}(1/2 + s - iT)
\\
+ 2\rho^*(1)\sum_{n \geq 1} \frac{\tau_{iT}(n)
}{n^{1/2 + s}} \frac{1}{2 \pi i} \int_{(\nu)} \widetilde{\psi}(-u) n^{-u}
\gamma_{V_T}(1/2 + s+u) du,
\end{multline}
where $\nu + \text{Re}(s) > 1/2$. Here and elsewhere, we denote by $\int_{(\nu)} f(u) du$ the contour integration along the contour $\{\nu+it \in \mathbb{C}:-\infty<t<+\infty\}$, i.e.,
\[
\int_{(\nu)} f(u) du=\int_{-\infty}^\infty f(\nu+it) dt.
\]
Let $T_0 = T^{1-\eta}$ where $\eta > 0$ is small.

\begin{lemma}[Trivial bound beyond the transition range]
\label{lemma:FboundTrivial}
 For $|t| \geq T + T_0$, and $-3<\sigma <3$, we have
 \begin{equation*}
  F(\sigma+it) \ll  (|t|T)^{-100}.
 \end{equation*}
\end{lemma}
\begin{proof}
 Using \eqref{eq:psiMellinbound} and the fact that $A$ is a small power of $T$ gives a satisfactory bound from the constant terms of $F(\sigma+it)$.  For the non-constant terms, we will have need of Stirling's bound
 \begin{equation}
 \label{eq:gammaVTStirlingBound}
  |\gamma_{V_T}(1/2+a + iv)| \ll_a (1 + |v+T|)^{\frac{a}{2}-\frac14}
  (1 + |v-T|)^{\frac{a}{2}-\frac14} \exp(-\tfrac{\pi}{4} (|v+T| + |v-T|)).
 \end{equation}
Letting $\text{Re}(u) = 1 - \sigma$, and using $\rho^*(1) = T^{o(1)} \exp(\frac{\pi T}{2})$, one sees that the exponential factor (including $\rho^*(1)$ and those from the gamma factors) is strictly negative for $v$ slightly beyond $T$.  Using \eqref{eq:psiMellinbound}, we may truncate the $u$-integral at $A T^{\varepsilon}$, which is much smaller than $T_0$, whence $|\text{Im}(s+u) \pm T| \geq T^{1/2}$, say, and so the exponential decay easily leads to a sufficient bound.
\end{proof}

Next we record another pointwise bound for $F$.
\begin{lemma}[Subconvexity-type pointwise bound for $F(it)$]
\label{lemma:FboundSubconvexity}
 Suppose $|t| \leq T - T_0$, and $-3<\sigma <3$.  Then
 \begin{equation*}
  F(\sigma + it) \ll (T T_0)^{-\frac{1}{12}+\varepsilon} A^{1/2}  \| \psi \|_2.
 \end{equation*}
\end{lemma}
\begin{proof}
 The constant term part of $F$ gives  $|\widetilde{\psi}(1/2 + \sigma+ it \pm i T)| \ll (A/T_0)^{100}$, which is acceptable.

 The non-constant terms give to $F(\sigma+it)$ the amount
 \begin{equation*}
   \frac{2 \rho^*(1)}{2 \pi i} \int_{(1)} \widetilde{\psi}(\sigma-u)
\gamma_{V_T}(1/2 + it+u) \zeta(1/2+it+u+iT) \zeta(1/2+it+u-iT) du.
 \end{equation*}
We move the contour of integration to $\nu = 0$, crossing poles at $u = 1/2 - it \mp iT$.  By \eqref{eq:psiMellinbound}, the residues at these points are small, since $|t\mp T| \geq T_0$, which is large compared to $A$.  The integral on the new line may be truncated at $|\text{Im}(u)| \ll A T^{\varepsilon}$, whence we obtain a bound
\begin{multline*}
 T^{\varepsilon} \int_{|v| \ll A T^{\varepsilon}} |\widetilde{\psi}(\sigma-iv)| \frac{|\zeta(1/2+it+iv+iT)|}{(1 + |t + v + T|)^{1/4} } \frac{|\zeta(1/2+it+iv-iT)|}{(1 + |t + v - T|)^{1/4} } dv
 \\
 \ll (T T_0)^{-\frac{1}{12} + \varepsilon} \int_{|v| \leq A T^{\varepsilon}} |\widetilde{\psi}(\sigma-iv)| dv,
\end{multline*}
using Weyl's subconvexity bound for $\zeta$.
The $v$-integral is estimated with \eqref{eq:psimellinL1boundSmalluRange}.
\end{proof}

Now we return to \eqref{eq:Ialphadef}.  Using a ``trivial" sup norm bound on the Eisenstein series in the form
\begin{equation}\label{eq:triv}
\sup_{z \in K} |E_T^*(z)| \ll_K T^{1/2+\varepsilon},
\end{equation}
where $K$ is a compact subset of $\mathbb{H}$, we have an easy trivial bound
\begin{equation*}
\int_0^{\infty} E_T^*\left(i\left(1+\frac{\alpha}{T}\right)y\right) \psi\left(\left(1+\frac{\alpha}{T}\right)y\right) y^{it} \frac{dy}{y} \ll \|\psi \|_1 T^{1/2 + \varepsilon}.
\end{equation*}
Combining this with Lemma \ref{lemma:FboundTrivial}, we derive that
\begin{equation*}
 I_{\psi,\alpha}(T) = \frac{1}{2\pi} \int_{|t| \leq T + T_0} F(-it) \left(\int_0^{\infty} E_T^*\left(i\left(1+\frac{\alpha}{T}\right)y\right) \psi\left(\left(1+\frac{\alpha}{T}\right)y\right) y^{it} \frac{dy}{y} \right) dt + O(T^{-10}).
\end{equation*}
In the inner $y$-integral, we change variables $y \rightarrow \left(1+\frac{\alpha}{T}\right)^{-1}y$, giving
\begin{align}
\label{eq:IpsialphaPreBinomialTheorem}
 I_{\psi,\alpha}(T) &= \frac{1}{2\pi} \int_{|t| \leq T + T_0} F(-it) \left(\int_0^{\infty} \left(1+\frac{\alpha}{T}\right)^{-it} E_T^*(iy) \psi(y) y^{it}  \frac{dy}{y} \right) dt + O(T^{-10})\notag \\
 &=\frac{1}{2\pi} \int_{|t| \leq T + T_0} \left(1+\frac{\alpha}{T}\right)^{-it}F(-it)F(it) dt + O(T^{-10}).
\end{align}

Next we state
\begin{lemma}[Bounding $t$ near $T_0$ with a fourth moment of $\zeta$]
\label{lemma:degeneraterange}
 We have
 \begin{equation*}
  \int_{|t \mp T| \leq T_0} |F(it)|^2 dt \ll  \|\psi^2 \|_1 + T^{\varepsilon} A \left(\frac{T_0^{1/2}}{T^{1/2}} + T^{-1/6} \right)  \|\psi^2\|_1.
 \end{equation*}
\end{lemma}
\begin{proof}
We treat the region $|t-T| \leq T_0$, since the other sign follows by symmetry.
The constant terms of $F$ satisfy
\begin{equation}
\label{eq:constanttermbound}
\frac{1}{2\pi} \int_{-\infty}^{\infty} |\widetilde{\psi}(1/2+  it + iT)|^2 dt = \int_0^{\infty} \psi(y)^2 \frac{dy}{y} =  \| \psi^2 \|_1,
\end{equation}
which is an acceptable error term.
For the non-constant terms, we have
\begin{multline*}
   \sum_{n \geq 1} \frac{\tau_{iT}(n)
}{n^{1/2 + it}} \frac{1}{2 \pi i} \int_{(1)} \widetilde{\psi}(-u) n^{-u}
\gamma_{V_T}(1/2 + it+u) du
\\=
\frac{1}{2 \pi i} \int_{(1)} \widetilde{\psi}(-u)
\zeta(1/2 + it + u +iT) \zeta(1/2+it+u-iT)
\gamma_{V_T}(1/2 + it+u) du.
\end{multline*}
We move the contour to the $0$-line, which crosses poles at $u+it = 1/2 \mp iT$.  A short calculation shows that these poles contribute a bound of the type \eqref{eq:constanttermbound}.  Therefore, we have
\begin{equation*}
 \int_{|t-T| \leq T_0} |F(it)|^2 dt \ll Q+  \| \psi^2 \|_1,
\end{equation*}
where
\begin{equation*}
 Q = \int_{|t-T| \leq T_0} |\rho(1)|^2 \Big| \int_{\mathbb{R}} \widetilde{\psi}(-iu)
\zeta(1/2 + it + iu +iT) \zeta(1/2+it+iu-iT)
\gamma_{V_T}(1/2 + it+iu) du \Big|^2 dt.
\end{equation*}
By squaring this out, and using Stirling, we derive
\begin{multline*}
 Q\ll \frac{1}{|\zeta(1+2iT)|^2} \int_{\mathbb{R}} \int_{\mathbb{R}} |\widetilde{\psi}(-iu_1) \widetilde{\psi}(iu_2) | \int_{|t-T| \leq T_0} \frac{|\zeta(1/2 + it + iu_1 + iT)|}{(1 + |t+u_1 + T|)^{1/4}}
 \\
 \frac{ |\zeta(1/2 + it+iu_1-iT)
 \zeta(1/2 - it - iu_2 + iT) \zeta(1/2 - it-iu_2-iT)|}
 { (1 + |t+u_1-T|)^{1/4} (1 + |t+u_2 + T|)^{1/4} (1 + |t+u_2-T|)^{1/4}} dt du_1 du_2.
\end{multline*}
Next we apply H\"{o}lder's inequality to the inner $t$-integral, with exponents $(4,4,4,4)$.  One of the integrals we need to bound is then
\begin{equation*}
H_1 =  \int_{|t-T| \leq T_0} \frac{|\zeta(1/2 + it + iu + iT)|^4}{(1 + |t+u+T|)} dt.
\end{equation*}
Since $\widetilde{\psi}(iu)$ is small for $|u| \geq A T^{\varepsilon}$, and $A$ is a small power of $T$, we have that $|t+u_1 + T| \asymp T$ in the relevant region of integration, and so by Iwaniec's bound (Proposition \ref{prop:Iwaniec}), we have
\begin{equation*}
 H_1 \ll T^{-1+\varepsilon} (T_0 + T^{2/3})
\end{equation*}
We also need to bound
\begin{equation*}
 H_2 = \int_{|t-T| \leq T_0} \frac{|\zeta(1/2 + it + iu - iT)|^4}{(1 + |t+u-T|)} dt = \int_{|t| \leq T_0} \frac{|\zeta(1/2 + it + iu )|^4}{(1 + |t+u|)} dt.
\end{equation*}
Recall that $\widetilde{\psi}(iu)$ is very small for $|u| \gg A T^{\varepsilon}$, and that we have assumed $A \ll T_0^{1/3}$.  Therefore, we  may certainly restrict the integrals so that $|u| \ll T_0$. Then we can change variables $t \rightarrow t- u$ and extend the integral to a constant multiple of $T_0$, thereby showing $H_2 \ll T^{\varepsilon}$.
%
Therefore, we derive
\begin{equation*}
 Q \ll T^{\varepsilon} T^{-1/2} (T_0 + T^{2/3})^{1/2} \Big[ \int_{\mathbb{R} } |\widetilde{\psi}(-iu)| du \Big]^2.
\end{equation*}
Using \eqref{eq:psiMellinL1norm} completes the proof.
\end{proof}

Applying Lemma \ref{lemma:degeneraterange} to \eqref{eq:IpsialphaPreBinomialTheorem}, we have
\begin{equation}
\label{eq:IpsialphaPreBinomialTheorem2}
I_{\psi,\alpha}(T) = \frac{1}{2\pi} \int_{|t| \leq T - T_0} \left(1+\frac{\alpha}{T}\right)^{-it}F(-it)F(it) dt +
O(ET_1),
\end{equation}
where
\begin{equation}
\label{eq:ET1def}
ET_1 = \| \psi^2 \|_1 +  T^{\varepsilon} A \left(\frac{T_0^{1/2}}{T^{1/2}} + T^{-1/6} \right)  \|\psi^2\|_1.
\end{equation}

Now we take a detour from \eqref{eq:IpsialphaPreBinomialTheorem2} to separate the variables.  The binomial theorem states that if $|x| < 1$, then
\begin{equation*}
 (1+x)^u = \sum_{\ell=0}^{\infty} \binom{u}{\ell} x^{\ell}, \quad \text{where} \quad \binom{u}{\ell} = \frac{u(u-1) \dots(u-\ell+1)}{\ell!}.
\end{equation*}
We are interested in bounding the tail of this series when the imaginary part of $u$ is potentially very large.  In our application, $x = \frac{\alpha}{T}$, and $|\frac{u}{T}| \ll 1$, so $|xu| \ll 1$.
\begin{lemma}
Suppose that $|z| < \frac{1}{10}$, $u \in \mathbb{C}$, and $|uz| \leq C$, for some constant $C > 0$.  Then
\begin{equation}
\label{eq:binomialTailBound}
(1+z)^u = \sum_{\ell = 0}^{L} \binom{u}{\ell} z^{\ell} + O(10^{-L}),
\end{equation}
where the implied constant depends on $C$ only.
\end{lemma}
\begin{proof}
By Taylor's theorem, we have
 \begin{equation*}
  (1+z)^u = \sum_{\ell=0}^{L} \binom{u}{\ell} z^{\ell}  + R_{L}(z),
 \end{equation*}
 assuming $|z| < \frac{1}{10}$,
where
\begin{equation*}
 R_L(z) =  \frac{1}{2 \pi i} \int_{|w| = 10 |z|} \left(\frac{z}{w}\right)^{L+1} \frac{(1+w)^u}{w-z} dw.
\end{equation*}
Then by the triangle inequality, we have
\begin{equation*}
 |R_L(z)| \leq \frac{1}{2 \pi} 20 \pi |z| 10^{-L-1} \frac{1}{9|z|} \max_{|w| = 10 |z|} |(1+w)^u| \leq 10^{-L} \max_{|w| = 10 |z|} |(1+w)^u|.
\end{equation*}
Then we have $(1+w)^u = \exp(u \log(1+w)) = \exp(u (w + O(w^2))) = O(1)$, since we assume $|uz| \ll 1$, and $|w| = 10|z| < 1$.
\end{proof}

Applying \eqref{eq:binomialTailBound} to \eqref{eq:IpsialphaPreBinomialTheorem2},  we have
\begin{equation}
\label{eq:IpsialphaPostBinomialTheorem}
 I_{\psi,\alpha}(T) = \sum_{\ell=0}^{L} \left(\frac{\alpha}{T}\right)^\ell \frac{1}{2\pi} \int_{|t| \leq T - T_0} \binom{-it}{\ell} |F(-it)|^2 dt + O(ET_1)
 + O\Big(10^{-L} \int_{|t| \leq T - T_0} |F(-it)|^2 dt \Big).
\end{equation}
Using $|F(-it)| \leq \int_0^{\infty} |E_T(iy)| \psi(y) \frac{dy}{y} \ll T^{1/2+\varepsilon} \| \psi \|_1$, by \eqref{eq:triv}
the second error term in \eqref{eq:IpsialphaPostBinomialTheorem} may be bounded by
\begin{equation*}
 10^{-L} \int_{|t| \leq T - T_0} \left(\int_0^{\infty} |E_T(iy)| \psi(y) \frac{dy}{y}\right)^2 dt \ll 10^{-L} T^{2+\varepsilon} \| \psi^2 \|_1.
\end{equation*}
Choosing
\begin{equation*}
L = 10 \log T + O(1),
\end{equation*}
say, lets us absorb this second error term into $ET_1$.


Next we wish to use an approximation of the form $\binom{-it}{\ell} \approx \frac{(-it)^{\ell}}{\ell !}$, for $\ell \geq 1$ (this is an identity for $\ell = 0$, of course).  However, for $t$ small compared to $\ell$, this is certainly not a good approximation.  For this reason, we employ a trivial bound for $|t| \leq T^{\delta}$ for $\binom{-it}{\ell}$ and we use Lemma \ref{lemma:FboundSubconvexity}, giving
\begin{equation*}
 \int_{|t| \leq  T^{\delta}} \left|\binom{-it}{\ell} \right|  |F(-it)|^2 dt \ll \frac{T^{\delta \ell} 2^{\ell}}{\ell !} A T^{\delta-1/6 + \varepsilon} \| \psi^2 \|_1,
\end{equation*}
for $\ell \geq 1$. We therefore have
\begin{multline*}
I_{\psi,\alpha}(T) =
\frac{1}{2\pi} \int_{|t| \leq T- T_0} |F(-it)|^2 dt
+
\sum_{\ell=1}^{L} \left(\frac{\alpha}{T}\right)^\ell \frac{1}{2\pi }
 \int_{T^{\delta} \leq |t| \leq T - T_0} \binom{-it}{\ell} |F(-it)|^2 dt
\\
+ O(ET_1)
 + O(T^{-\frac{7}{6} + 2 \delta + \varepsilon} A \| \psi^2 \|_1).
\end{multline*}

Meanwhile, we have
\[
 \ell ! \binom{-it}{\ell} = (-it) \ldots (-it-\ell+1) = (-it)^{\ell} \left(1 + \frac{1}{it}\right) \ldots \left(1 + \frac{\ell-1}{it}\right).
\]
For this we use
\begin{align*}
 & \Big|\left(1 + \frac{1}{it}\right) \ldots \left(1 + \frac{\ell-1}{it}\right) - 1\Big|\\
 =& \Big|\frac{1}{it}(1 + \ldots + (\ell-1)) + \frac{1}{(it)^2}\Big(\sum_{1 \leq i < j \leq \ell-1} i j\Big) + \ldots + \frac{1}{(it)^{\ell-1}} (1 \cdot \ldots \cdot(\ell-1))\Big|\\
 \leq & \sum_{k=1}^{\ell-1} \frac{1}{|t|^{k}} \binom{\ell-1}{k}(\ell-1)^k
 \leq  \frac{\ell^2}{|t|} \sum_{k=0}^{\ell-2} \frac{\ell^{2k}}{|t|^{k}} .
\end{align*}
Thus, if $\frac{\ell^2}{|t|} <\frac{1}{2}$, we have
\[
\ell ! \binom{-it}{\ell} = (-it)^\ell \left(1+O\left(\frac{\ell^2}{|t|}\right)\right).
\]
Let us choose $\delta > 0$ so that $T^{\delta} \gg A^2$, and then
\begin{multline*}
T^{-\ell} \int_{T^{\delta} \leq |t| \leq T - T_0} \binom{-it}{\ell} |F(-it)|^2 dt = T^{-\ell} \int_{T^{\delta} \leq |t| \leq T - T_0} \frac{(-it)^{\ell}}{\ell !} |F(-it)|^2 dt
\\
+O\left(\frac{\ell^2}{T \ell !} \int_{T^{\delta} \leq |t| \leq T - T_0}  |F(-it)|^2 dt \right).
\end{multline*}
After using this approximation, we may safely re-extend the integrals to include $|t|\leq T^{\delta}$, without making a new error term.

With Proposition \ref{prop:easybound} below, we will show
\begin{equation*}
 \int_{|t| \leq T- T_0} |F(-it)|^2 dt \ll  T^{\varepsilon} A \| \psi^2 \|_1,
\end{equation*}
using relatively simple methods.  Taking this for granted, we then derive
\begin{lemma}
 Assume $A \ll T^{1/10}$, and $T_0 = T^{1-\eta}$ with some $0 < \eta \leq \frac{1}{10}$.  Then
 \begin{equation}
 \label{eq:IpsialphaPostBinomialTheorem3a}
  I_{\psi,\alpha}(T) = \sum_{\ell=0}^{L} \frac{(-i\alpha)^\ell}{\ell !}   \frac{1}{2\pi } \int_{|t| \leq T - T_0} \left(\frac{t}{T}\right)^{\ell} |F(-it)|^2 dt +
O(ET_1) + O(ET_2),
 \end{equation}
 where
 \begin{equation}
 \label{eq:ET2def}
 ET_2 = AT^{-1+\varepsilon} \|\psi^2\|_1 + A^{} T^{2 \delta - \frac16 - 1} \| \psi ^2\|_1.
 \end{equation}
\end{lemma}

Let
\begin{equation*}
 W_{T}(n,t) = \frac{1}{2 \pi i} \int_{(a)} \widetilde{\psi}(-u) n^{-u}
\frac{\gamma_{V_T}(1/2 + it+u)}{\gamma_{V_T}(1/2+it)} du,
\end{equation*}
where $a > 1/2$.
\begin{lemma}
\label{lemma:weightfunction}
 Suppose that $|t| \leq T - T_0$, where $T_0 = T^{1-\eta}$ with $0 < \eta \leq \frac{1}{10}$.  Then
 \begin{equation}
 \label{eq:WellTasymptotic}
  W_{T}(n,t) = \psi\left(\frac{\sqrt{T^2-t^2}}{2\pi n} \right) +
O_{\delta}(A^3 n^{-1/2-\delta} T^{1/4 + \delta/2} T_0^{-3/4 + \delta/2}),
 \end{equation}
where $0 < \delta < 1$.
\end{lemma}
\begin{proof}

Using \eqref{eq:psiMellinbound} and \eqref{eq:gammaVTStirlingBound} shows that the tail of the integral with $|\text{Im}(u)| \gg A T^{\varepsilon}$ is bounded by
\begin{equation*}
\int_{|v| \gg A T^{\varepsilon}} T^{-1000} (1+|v|)^{-100} n^{-a} \frac{(1+|t+v+T|)^{\frac{a}{2}-\frac14}}{|t+T|^{-1/4}} \frac{(1+|t+v-T|)^{\frac{a}{2}-\frac14}}{|t-T|^{-1/4}} \exp(-\tfrac{\pi}{4} Q(v,t,T)) dv,
\end{equation*}
where $Q(v,t,T) = |t+T+v| + |t-T+v| - |t+T| - |t-T|$.  It is easy to see that  in terms of $v$, $Q$ is minimized on the interval $[-t-T, -t+T]$ (note $Q$ is constant on this interval), in which case $Q(v,t,T) = 0$, since $|t| \leq T$.  Therefore, it is easy to see that the tail is much smaller than the displayed error term in \eqref{eq:WellTasymptotic}, assuming $1/2 < a  < 3/2$, since the integrand is exponentially small.

In the range $|t|  \leq T - T_0$, $|\text{Im}(u)| \leq A T^{\varepsilon}$, Stirling gives
\begin{equation}
\label{eq:Stirling}
 \Gamma\left(\frac{1/2 + it + u \pm iT}{2}\right) = \Gamma\left(\frac{1/2 + it \pm iT}{2}\right)
 \left(\frac{1/2 + it \pm iT}{2} \right)^{u/2} \left(1 + O\left(\frac{(1+|u|^2)}{|t\mp T|} \right) \right).
\end{equation}
Note that $|u|^2 \ll A^2 T^{2\varepsilon}$, while $|t \mp T| \geq T_0$, so this error term is indeed smaller than $1$, and it is acceptable to truncate the asymptotic expansion at this point.

In this way, one derives from \eqref{eq:gammaVformula} and \eqref{eq:Stirling} the asymptotic
\begin{equation*}
\frac{\gamma_{V_T}(1/2 + it + u)}{\gamma_{V_T}(1/2 + it) } = (2\pi)^{-u} |t+T|^{\frac{u}{2}} |t-T|^{\frac{u}{2}}
\left(1 + \sum_{\pm} O\left(\frac{(1+|u|^2)}{|t\mp T|} \right) \right).
\end{equation*}
Therefore,
\[
W_{T}(n,t) =
\frac{1}{2 \pi i} \int_{\substack{\text{Re}(u) = a \\ |\text{Im}(u)| \ll A T^{\varepsilon}}} \widetilde{\psi}(-u)
\left(\frac{\sqrt{T^2-t^2}}{2 \pi n} \right)^u
 \left(1 + \sum_{\pm} O\left(\frac{(1+|u|^2)}{|t\mp T|} \right) \right) du + O(n^{-1} T^{-100}).
\]

The leading term here gives $\psi(\frac{\sqrt{T^2-t^2}}{2\pi n})$, after extending the integral to include $|\text{Im}(u)| \gg A T^{\varepsilon}$,
since
\begin{equation*}
\frac{1}{2 \pi i} \int_{(a)} \widetilde{\psi}(-u) x^u du =  \psi(x).
\end{equation*}
The error term arising from Stirling is
\begin{equation*}
\ll (T-t)^{\frac{a}{2}-1} T^{\frac{a}{2}} n^{-a} \int_{-\infty}^{\infty} |\widetilde{\psi}(iv)| (1 + v^2) dv.
\end{equation*}
The inner integral above is $\ll A^3$, and so taking $a = 1/2 + \delta$ completes the proof.
\end{proof}

Applying Lemma \ref{lemma:weightfunction} to \eqref{eq:Fformula1},  using \eqref{eq:psiMellinbound} to estimate the constant terms,
and using $|\rho^*(1) \gamma_{V_T}(1/2+it)| \ll T^{-1/4+\varepsilon} T_0^{-1/4}$,
we derive a pointwise approximation to $F$ of the form
\begin{equation}
\label{eq:FGapproximation}
F(it ) =
G(it) +
O_{\varepsilon}\left(\frac{A^3 T^{\varepsilon}}{ T_0^{}}\right),
\end{equation}
where
\begin{equation}
\label{eq:Gdef}
G(it) = 2 \rho^*(1) \gamma_{V_T}(1/2 + it) \sum_{n=1}^{\infty} \frac{\tau_{iT}(n)}{n^{1/2+it}} \psi\left(\frac{\sqrt{T^2-t^2}}{2\pi n}\right).
\end{equation}

Defining
\begin{equation*}
J_{\ell} = \frac{1}{2\pi} \int_{|t| \leq T-T_0} \frac{t^{\ell}}{T^{\ell}} |F(it)|^2 dt,
\quad \text{and} \quad
K_{\ell} = \frac{1}{2\pi} \int_{|t| \leq T-T_0}  \frac{t^{\ell}}{T^{\ell}} |G(it)|^2 dt,
\end{equation*}
we obtain from \eqref{eq:FGapproximation} and Cauchy-Schwarz that
\begin{equation}
\label{eq:IellInitialApproximation}
J_{\ell} = K_{\ell} + O(K_{\ell}^{1/2} A^3 T_0^{-1} T^{1/2 + \varepsilon}) + O( A^6 T_0^{-2} T^{1+2 \varepsilon} ).
\end{equation}

Our goal is now to find an asymptotic for $K_{\ell}$.  In order to simplify the error term in \eqref{eq:IellInitialApproximation}, we claim the following
\begin{proposition}[Large sieve-type bound]
\label{prop:easybound}
 We have
 \[
  \int_{|t| \leq T-T_0} |G(it)|^2 dt \ll T^{\varepsilon} A \| \psi^2 \|_1.
 \]
The same bound holds with $G(it)$ replaced by $F( it)$, provided $A \ll T^{1/10}$ and $T_0 \gg T^{9/10}$.
\end{proposition}
We defer the proof to Section \ref{section:largesieve}.  For now, we simplify the error term as follows.  Note that \eqref{eq:psiintegrallowerbound} implies $K_{\ell}^{1/2} \ll A T^{\varepsilon} \| \psi^2 \|_1$, and $1 \ll A \| \psi^2 \|_1$.  Using these, and applying \eqref{eq:IellInitialApproximation} to \eqref{eq:IpsialphaPostBinomialTheorem3a}, we deduce
\begin{equation}
 \label{eq:IpsialphaPostBinomialTheorem3b}
  I_{\psi,\alpha}(T) = \sum_{\ell=0}^{L} \frac{(-i\alpha)^\ell}{\ell !}   K_{\ell} +
O(ET_1) + O(ET_2) + O(ET_3),
 \end{equation}
under the hypotheses of Proposition \ref{prop:easybound}, where
\begin{equation}
\label{eq:ET3def}
 ET_3 = (T_0^{-1} A^4 T^{1/2 + \varepsilon} + T_0^{-2} A^7 T^{1+\varepsilon}) \|\psi^2 \|_1.
\end{equation}


Next we develop some analytic properties of the weight function appearing in \eqref{eq:Gdef}.
Define
\begin{equation}
\label{eq:Psidef}
\Psi_{T}(x,t) =  \psi\left(\frac{\sqrt{T^2-t^2}}{2\pi x} \right).
\end{equation}
We need to understand derivatives of $\Psi_{T}$ in terms of both $x$ and $t$.  We shall build this up from simpler estimates.  

With shorthand $Y = (T^2-t^2)^{1/2}/(2\pi)$, we have from Faa di Bruno's formula
\begin{multline*}
 \frac{d^n}{dx^n} \psi\left(\frac{Y}{x}\right)
 \\
 = \sum_{m_1 + 2m_2 + \dots + n m_n = n} \frac{n!}{m_1! (1!)^{m_1} \dots m_n! (n!)^{m_n}} \psi^{(m_1 + \dots + m_n)}\left(\frac{Y}{x}\right) \left(\frac{1! Y}{x} \frac{1}{-x}\right)^{m_1} \dots \left(\frac{n! Y}{x} \frac{1}{(-x)^n}\right)^{m_n},
\end{multline*}
which leads to the bound
\begin{equation*}
 \frac{d^n}{dx^n} \psi\left(\frac{Y}{x}\right) \ll \sum_{m_1 + 2m_2 + \dots + n m_n = n} \frac{n!}{m_1!  \dots m_n! } \left(\frac{Y}{x}\right)^{m_1 + \dots + m_n} \psi^{(m_1 + \dots + m_n)}\left(\frac{Y}{x}\right) \frac{1}{x^{n}} \ll \left(\frac{A}{x} \right)^n.
\end{equation*}
By a similar calculation, we have
\begin{equation*}
 \frac{d^n}{dt^n} \psi\left(\frac{\sqrt{T^2-t^2}}{2 \pi x} \right) \ll \left(\frac{A}{T-|t|} \right)^n.
\end{equation*}
These bounds immediately imply
\begin{equation}
\label{eq:PsiDerivativeBoundsWithxAndt}
\frac{\partial^{j}}{\partial t^j} \Psi_{T}(x,t) \ll (T- |t|)^{-j} A^{j},
\quad
\frac{\partial^{j}}{\partial x^j} \Psi_{T}(x,t) \ll x^{-j} A^{j},
\end{equation}
using \eqref{eq:psibound} and that $\psi$ is supported on $[1,3]$.

It may also be useful to record that
\begin{equation*}
\int_0^{\infty} \Psi_{T}(x,t)^2 x^v \frac{dx}{x} = \left(\frac{T^2-t^2}{4\pi^2}\right)^{v/2} \widetilde{\psi^2}(-v) \ll_{\text{Re}(v)} (T^2-t^2)^{\text{Re}(v)/2} 
\| \psi^2 \|_1.
\end{equation*}

If we let $\Phi_{t,T}(x) = \Psi^2_{T}(x,t)$, and let $\widetilde{\Phi}_{t,T}$ denote its Mellin transform (with respect to the $x$-variable), then we have
\begin{equation*}
 \frac{d^n}{dx^n} \Phi_{t,T}(x) \ll \left(\frac{A}{x}\right)^n,
\end{equation*}
and therefore if $\Phi$ is supported on $x \asymp X$, then
\begin{equation}
\label{eq:PhiMellinBound}
 \widetilde{\Phi}_{t,T}(v) \ll_{\text{Re}(v)} X^{|\text{Re}(v)|} \left(1 + \frac{|v|}{A}\right)^{-C}.
\end{equation}

\section{Shifted convolution sum approach}
\subsection{Statement of result, and smooth partition of unity}
The purpose of this section is to show
\begin{proposition}
\label{prop:bulklt}
Suppose that $\ell \leq L \ll \log T \ll A \ll T^{1/10}$, and $T_0 = T^{1-\eta}$ with $0 < \eta \leq 1/10$.  Then
we have
\begin{multline*}
K_{\ell}=
\frac{6}{\pi} \log(1/4 + T^2) \left(c_{\ell} + O((\log T)^{-1/3+\varepsilon})\right)\| \psi^2 \|_1
\\
+
  O(T^{-\frac{5}{12} + \varepsilon} T_0^{\frac14} A^{\frac12} + T_0^{-1} A^2 T^{\varepsilon} )
+O(   A^3 T_0^{-\frac{9}{4}} T^{\frac{25}{12}+\varepsilon} + A^{\frac32} T_0^{-\frac{5}{8}} T^{\frac{13}{24}+\varepsilon}).
\end{multline*}
where $c_{\ell} =0$ for $\ell$ odd, and for $\ell$ even,
 \begin{equation*}
 c_{\ell} =  \frac{\Gamma(\frac{1+\ell}{2})}{\Gamma(\frac12) \Gamma(1+\frac{\ell}{2})}.
 \end{equation*}
\end{proposition}
\begin{remark}
The key calculation for evaluating the sum over $\ell$ of the main term is that
\begin{equation}
\label{eq:cellsumevaluation}
\sum_{\ell \text{ even}} \frac{(-i \alpha)^{\ell}}{\ell !} c_{\ell} = J_0(\alpha).
\end{equation}
To see this, we use a gamma function identity to get
\begin{equation*}
\sum_{n=0}^{\infty} \frac{(-1)^n \alpha^{2n}}{(2n)!} \frac{\Gamma(\frac{1+2n}{2})}{\Gamma(\frac12) \Gamma(1+n)} = \sum_{n=0}^{\infty} (-1)^n \frac{(\frac{\alpha}{2})^{2n}}{n!} \frac{1}{\Gamma(1+n)} = J_0(\alpha).
\end{equation*}
\end{remark}

We apply a smooth partition of unity, obtaining
\begin{equation*}
\int_{|t| \leq T-T_0} \left(\frac{t}{T}\right)^{\ell}  |G(it)|^2 \frac{dt}{2\pi} = \sum_{\Delta, \pm} I_{\pm \Delta} + O(I_{T_0}), \quad I_{\pm \Delta}  =  \int_{-\infty}^{\infty} w_{\pm \Delta}(t) \left(\frac{t}{T}\right)^{\ell}  |G(it)|^2 \frac{dt}{2\pi},
\end{equation*}
where we choose the partition such that each $w_{\pm \Delta}$ is supported on $|t\mp T| \asymp \Delta$, for $T_0 \ll \Delta \ll T$, except for one constituent of the partition which is supported on $|t| \leq T/2$, say.   The error term arises because the sharp truncation $|t| \leq T- T_0$ is not smooth, but instead we may over-extend the integral with $|t-T| \asymp T_0$ by multiplying by a smooth function; this gives rise to $I_{T_0}$.
In fact we choose $w$ of the form $w_{\pm \Delta}(t) = w(\frac{T \mp t}{\Delta} )$ (except for the constituent with $|t| \leq T/2$).

By squaring out, we derive
\begin{multline*}
 I_{\pm \Delta} = \frac{4 |\rho^*(1)|^2}{\cosh(\pi T)}
 \sum_{m,n=1}^{\infty} \frac{\tau_{iT}(m) \tau_{iT}(n)}{\sqrt{mn}}
 \\
 \int_{-\infty}^{\infty} \cosh(\pi T) \left(\frac{t}{T}\right)^{\ell} |\gamma_{V_T}(1/2+it)|^2 \Psi_{ T}(m,t) \Psi_{ T}(n,t) w_{\Delta}(t) \left(\frac{m}{n}\right)^{it} \frac{dt}{2\pi}.
\end{multline*}
Define $N := \sqrt{\Delta T}$
and observe that the support of $\Psi$ means that $m,n \asymp N$.

\subsection{Proof of Proposition \ref{prop:easybound}}
\label{section:largesieve}
 This follows almost immediately from the mean value theorem for Dirichlet polynomials, which states that $\int_0^{T} |\sum_{n \leq N} a_n n^{-it}|^2 \ll (N+ T) \sum_{n \leq N} |a_n|^2$, for an arbitrary sequence $a_n$.  The only subtlety is that in our desired application, the coefficient $a_n$ depend slightly on $t$.  One can remove this dependence by separation of variables in many possible ways.  We will use a variation of \cite[Lemma 4.2]{YoungPreprint}.  By this type of reasoning, we have
 \begin{equation*}
  I_{\pm \Delta} \ll \frac{T^{o(1)}}{(\Delta T)^{1/2}} \int_{|t-T| \ll \Delta} \Big| \int_{(0)} \widetilde{\psi}(-u) \left(\frac{\sqrt{T^2-t^2}}{2\pi } \right)^u \sum_{n} \frac{\tau_{iT}(n)}{n^{1/2+it+u}} du \Big|^2 dt.
 \end{equation*}
By the support of $\psi$, we have that $n \asymp (T^2-t^2)^{1/2} \ll (\Delta T)^{1/2}$, a redundant property that we enforce manually.
We may truncate the $u$-integral at $|u| \ll A T^{\varepsilon}$ and then apply Cauchy-Schwarz to the inner $u$-integral, obtaining
\begin{equation*}
 I_{\pm \Delta} \ll \frac{T^{\varepsilon}}{(\Delta T)^{1/2}} \| \psi^2 \|_1 \int_{|u| \ll A T^{\varepsilon}} \int_{|t-T| \ll \Delta} \Big|\sum_{n \ll (\Delta T)^{1/2}} \frac{\tau_{iT}(n)}{n^{1/2+it+u}}  \Big|^2 dt du + O(T^{-100}).
\end{equation*}
Now we can apply the mean value theorem for Dirichlet polynomials to the $t$-integral, and integrate trivially over $u$.

The analogous bound for $F$ would follow by repeating the same argument, including lower-order terms arising from Stirling.  For simplicity, we shall only use \eqref{eq:FGapproximation}, giving
\begin{equation*}
\int_{|t| \leq T-T_0} |F(it)|^2 dt \ll \int_{|t| \leq T-T_0} |G(it)|^2 dt
+ \frac{A^6 T}{T_0^2} T^{\varepsilon}.
\end{equation*}
Using $A \| \psi^2\|_1 \gg 1$, and $\frac{A^6 T}{T_0^2} \ll T^{\frac{6}{10}+1-\frac{18}{10}} = T^{-\frac{2}{10}}$ shows that this secondary term satisfies a (more than) satisfactory bound.

\subsection{Evaluation of the diagonal term}
Let $I_{\pm \Delta}^{\text{diag}}$ denote the terms with $m=n$, and let $I^{\text{diag}} = \sum_{\pm, \Delta} I_{\pm \Delta}^{\text{diag}}$.  We first focus on $I^{\text{diag}}$.
\begin{proposition}
 We have
 \begin{equation*}
  I^{\text{diag}} = \frac{6}{\pi} \log(1/4 + T^2) \left(c_{\ell} + O((\log T)^{-1/3+\varepsilon})\right)\| \psi^2 \|_1  +
  O(T^{-\frac{5}{12} + \varepsilon} T_0^{\frac14} A^{\frac12} + T_0^{-1} A^2 ).
 \end{equation*}
\end{proposition}
This is the main term stated in Proposition \ref{prop:bulklt}, with a compatible error term.

\begin{proof}
The diagonal term takes the form
\begin{equation*}
I^{\text{diag}} = 4 \frac{|\rho^*(1)|^2}{2\pi} \int_{|t| \leq T- T_0} \left(\frac{t}{T}\right)^{\ell}  |\gamma_{V_T}(1/2 + it)|^2
\sum_{n=1}^{\infty} \frac{\tau_{iT}(n)^2}{n} \Psi_{T}(n,t)^2 dt.
\end{equation*}
We shall asymptotically evaluate the inner sum over $n$, and then perform the $t$-integral afterwards.

Recall $\Psi_{T}(n,t)^2 = \Phi_{t, T}(n)$, and the bound \eqref{eq:PhiMellinBound}; in our case $X \asymp (T^2-t^2)^{1/2}$.

We have by Mellin inversion and \eqref{eq:Zformula} that
\begin{equation*}
 \sum_{n=1}^{\infty} \frac{\tau_{iT}(n)^2}{n} \Psi_{T}(n,t)^2 = \frac{1}{2 \pi i} \int_{(1)} \widetilde{\Phi}_{t,T}(v) \frac{\zeta(1+v)^2 \zeta(1+v+2iT) \zeta(1+v-2iT)}{\zeta(2+2v)} dv.
\end{equation*}
Moving the contour to the line $\text{Re}(v) = -1/2$, we obtain an error term of size
\begin{equation*}
 \ll (T^2 - t^2)^{-1/4}  \int_{|v| \leq A T^{\varepsilon}} |\zeta(1/2 + iv)|^2 |\zeta(1/2 + iv + iT)\zeta(1/2 + iv-iT)| dv,
\end{equation*}
plus an error from the truncation that will be dwarfed by the upcoming error term.  We may use Cauchy-Schwarz and Proposition \ref{prop:Iwaniec} to give that this is in turn
\begin{equation*}
 \ll T^{\varepsilon}  (T^2 - t^2)^{-1/4} A^{1/2} (T^{2/3} + A)^{1/2}  \ll T_0^{-1/4} T^{1/12 + \varepsilon} A^{1/2},
\end{equation*}
under the assumption $A \ll T^{2/3}$, which is valid.
In turn, this error contributes to $I^{\text{diag}}$ at most
\begin{equation*}
 \ll T_0^{-\frac14} T^{\frac{1}{12} + \varepsilon} A^{\frac12}  \int_{|t| \leq T-T_0} (1 + |t-T|)^{-1/2} (1 + |t+T|)^{-1/2} dt  \ll T^{-\frac{5}{12} + \varepsilon} T_0^{\frac14} A^{\frac12} .
\end{equation*}

Next we need to analyze the residue at $v=0$ (the ones at $v= \pm 2iT$ are small by the decay of $\widetilde{\Phi}$).  By a direct calculation, we have that the residue equals
\begin{equation*}
 \frac{|\zeta(1+2iT)|^2}{\zeta(2)} \left( \widetilde{\Phi}_{t,T}'(0) + \widetilde{\Phi}_{t, T}(0) (c + \frac{\zeta'}{\zeta}(1+2iT)) \right),
\end{equation*}
for some absolute constant $c$.  We have
\begin{equation*}
 \widetilde{\Phi}_{t,T}(0) = \int_0^{\infty} \Psi_{T}(x,t)^2 \frac{dx}{x}, \qquad \widetilde{\Phi}_{t,T}'(0) = \int_0^{\infty} \Psi_{T}(x,t)^2 \log x \frac{dx}{x}.
\end{equation*}
Using \eqref{eq:Psidef} and changing variables $x \rightarrow x^{-1} \frac{(T^2-t^2)^{1/2}}{2\pi}$, we derive that
\begin{equation*}
 \widetilde{\Phi}_{t,T}'(0) = (\log (\sqrt{T^2-t^2}) + O(1))\widetilde{\Phi}_{t,T}(0).
\end{equation*}
Using the Vinogradov-Korobov bound \eqref{eq:VinogradovKorobovZetalogarithmicderivative}, and the fact that $\log (T^2 -t^2) \gg \log T$, we derive that the residue is
\begin{equation*}
 \frac{|\zeta(1+2iT)|^2 \widetilde{\Phi}_{t,T}(0)}{\zeta(2)} \left( \log( \sqrt{T^2-t^2}) + O( (\log T)^{2/3 + \varepsilon}) \right).
\end{equation*}
In fact, by a change of variables, we see that $\widetilde{\Phi}_{t,T}(0)$ is equal to $\|\psi^2\|_1$.
Thus we have shown that the residue contributes to $I^{\text{diag}}$
\begin{equation*}
\frac{4|\rho^*(1)|^2 |\zeta(1+2iT)|^2 \|\psi^2\|_1}{2\pi \zeta(2)} \int_{|t| \leq T- T_0} \left(\frac{t}{T}\right)^{\ell}   |\gamma_{V_T}(1/2 + it)|^2  \log(\sqrt{T^2-t^2}) (1 + O((\log T)^{-1/3 + \varepsilon})) dt.
\end{equation*}

For the evaluation of the $t$-integral,  we claim
\begin{equation*}
 \int_{|t| \leq T-T_0} \left(\frac{t}{T}\right)^{\ell} \log(\sqrt{T^2-t^2}) |\gamma_{V_T}(1/2 + it)|^2 dt = \frac{\pi^2}{ \cosh(\pi T)} \log T \frac{\sqrt{\pi}}{2} \frac{\Gamma(\frac{1+\ell}{2})}{\Gamma(1+\frac{\ell}{2})} + O(e^{- \pi T}),
\end{equation*}
for $\ell$ even.  Of course, this integral vanishes for $\ell$ odd.  We derive this now.
Stirling's approximation gives
\begin{equation*}
 |\gamma_{V_T}(1/2+it)|^2 = \frac{ \pi^2}{2 \cosh(\pi T)} (T^2-t^2)^{-1/2} (1 + O(T_0^{-1})).
\end{equation*}
Changing variables $t = Tv$ followed by $v = \sin(u)$ gives
\begin{multline*}
 \int_{|t| \leq T-T_0} \left(\frac{t}{T}\right)^{\ell} \log(\sqrt{T^2-t^2}) |\gamma_{V_T}(1/2 + it)|^2 dt
\\
= \frac{\pi^2}{ \cosh(\pi T)} \int_0^{\frac{\pi}{2}-\varepsilon_0} (\sin u)^{\ell} (\log T + \log \cos u) (1 + O(T_0^{-1})) du,
\end{multline*}
where $\varepsilon_0$ is defined implicitly by the relation $\sin(\frac{\pi}{2}-\varepsilon_0) = 1 - \frac{T_0}{T}$.  By a Taylor expansion, we see $\varepsilon_0 \sim \sqrt{2T_0/T}$.  The part of the integral with $\log \cos u$ may be bounded trivially since this function is integrable at $\pi/2$.  We may also extend the integral to $\pi/2$, thereby getting
\begin{equation*}
\frac{\pi^2}{ \cosh(\pi T)} \log T \int_0^{\frac{\pi}{2}} (\sin u)^{\ell} du + O(e^{-\pi T}).
\end{equation*}
We also have from \cite[(3.621.1)]{GR} that
\begin{equation*}
\int_0^{\frac{\pi}{2}} (\sin u)^{\ell} du = \frac{\sqrt{\pi}}{2} \frac{\Gamma(\frac{1+\ell}{2})}{\Gamma(1+\frac{\ell}{2})},
\end{equation*}
which is the last step required for the proof of the claim.

Therefore, the residue gives to $I^{\text{diag}}$
\begin{equation*}
 \frac{4|\rho^*(1)|^2 |\zeta(1+2iT)|^2 \|\psi^2\|_1 }{2\pi \zeta(2)} \frac{\pi^2 \log T}{\cosh(\pi T)} \frac{\sqrt{\pi}}{2} \frac{\Gamma(\frac{1+\ell}{2})}{\Gamma(1+\frac{\ell}{2})} + O(1).
\end{equation*}
One then easily checks, using the formulas following \eqref{eq:ETFourier}, that
\begin{equation*}
 \frac{4|\rho^*(1)|^2 |\zeta(1+2iT)|^2 }{2\pi \zeta(2)} \frac{\pi^2 \log T}{\cosh(\pi T)} \frac{\sqrt{\pi}}{2} \Gamma(1/2) = \frac{12}{\pi} \log T. \qedhere
\end{equation*}

\end{proof}

\subsection{Off-diagonal terms analysis}
Let
\begin{equation*}
I_{\Delta}^{OD} = \frac{c}{|\zeta(1+2iT)|^2}
 \sum_{m \neq n} \frac{\tau_{iT}(m) \tau_{iT}(n)}{\sqrt{mn}}
 K(m,n),
\end{equation*}
where $c = 4/\pi^2$, and
\begin{equation*}
 K(m,n) = \int_{-\infty}^{\infty} \cosh(\pi T) \gamma_{V_T}(1/2+it)|^2 \left(\frac{t}{T}\right)^{\ell} \Psi_{T}(m,t) \Psi_{T}(n,t) w_{\Delta}(t) \left(\frac{m}{n}\right)^{it} dt.
\end{equation*}
We focus on the case $I_{+ \Delta} = I_{\Delta}$ since the opposite sign case is estimated in the same way by conjugation.
In this section, we show
\begin{equation}
\label{eq:RdeltaBound}
 I_{\Delta}^{OD} = M.T.^{(\Delta)}_{OD} + O(T^{\varepsilon} ( A^3 \Delta^{-\frac{9}{4}} T^{\frac{25}{12}} + A^{\frac{3}{2}} \Delta^{-\frac{5}{8}} T^{\frac{13}{24}})),
\end{equation}
where $M.T.^{(\Delta)}_{OD}$ is a certain main term analyzed in Section \ref{section:ODMT}.
The smallest value of $\Delta$ is $T_0$, which leads to two of the error terms appearing in Proposition \ref{prop:bulklt}.

We need to estimate $K(m,n)$.
One can quickly check using Stirling's formula that for $|t|\leq T - T_0$, we have
\begin{equation*}
\frac{d^j}{dt^j} \cosh(\pi T) |\gamma_{V_T}(1/2 + it)|^2 \ll
 T^{-1/2} (T-|t|))^{-k-1/2}, \quad k=0,1,2,\dots.
\end{equation*}
Combined with \eqref{eq:PsiDerivativeBoundsWithxAndt}, this shows that
\begin{equation}
\label{eq:PsiPsiderivativebound}
\frac{\partial^j}{\partial t^j} \left(\frac{t}{T}\right)^{\ell} \Psi_{ T}(m,t) \Psi_{ T}(n,t) \cosh(\pi T) |\gamma_{V_T}(1/2 + it)|^2 w_{\Delta}(t) \ll T^{-1/2} \Delta^{-1/2} \left(\frac{A}{\Delta}\right)^j,
\end{equation}
where when dealing with derivatives of $(t/T)^{\ell}$, we have used that $\frac{\ell}{T} \ll \frac{1}{\Delta}$.

In light of \eqref{eq:PsiPsiderivativebound}, we may view $K$ as the Fourier transform of a function with controlled derivatives.  By a  standard integration by parts argument, we have that
\begin{equation*}
 K(m,n) \ll (\Delta T)^{-1/2} \Delta \left(1 + \frac{\Delta}{A} |\log(m/n)| \right)^{-C},
\end{equation*}
where $C > 0$ is arbitrary.  Therefore, we may assume that $m= n+h$ where $|h| \ll \frac{NA}{\Delta} T^{\varepsilon}$.
Furthermore, for the relevant values of $h$, we have
\begin{equation*}
\frac{d^j}{dx^j} K(x+h, x) \ll (\Delta T)^{-1/2} \Delta T^{\varepsilon} \left(\frac{AT}{N \Delta}\right)^j.
\end{equation*}
Then we may apply Proposition \ref{prop:SCS}, with $Y= N$, $P = \frac{AT}{\Delta}$, $R \asymp P$, and $M = \frac{NA}{\Delta} T^{\varepsilon}$, which gives
\begin{equation*}
R_{\Delta} = M.T.^{(\Delta)}_{OD} + E.T.,
\end{equation*}
where
\begin{equation*}
E.T. \ll (\Delta T)^{-\frac12} \Delta N^{-1} T^{\varepsilon} \left(\frac{NA}{\Delta}\right)
\left(T^{\frac13} N^{\frac12} R^{2} + T^{\frac16} N^{\frac34} R^{\frac12} \right).
\end{equation*}
Substituting for $R$ and $N$ and simplifying, we have
\begin{equation*}
E.T. \ll  T^{\varepsilon} ( A^3 \Delta^{-\frac94} T^{\frac{25}{12}} + A^{\frac{3}{2}} \Delta^{-\frac{5}{8}} T^{\frac{13}{24}}),
\end{equation*}
consistent with \eqref{eq:RdeltaBound}.

\subsection{Off-diagonal main term}
\label{section:ODMT}
Next we estimate the main term $M.T.^{(\Delta)}_{OD}$, which turns out to be a rather subtle problem.  Some of the ideas used here were initially developed in \cite{YoungPreprint}.
\begin{lemma}
\label{lemma:MTODbound}
Assuming $\ell \ll \log T$, we have
 \begin{equation}
 \label{eq:MTODbound}
\sum_{\substack{T_0 \ll \Delta \ll T \\ \Delta \text{dyadic}}}  M.T.^{(\Delta)}_{OD} \ll (\log \log T) \| \psi^2 \|_1 + T_0^{-2} A^3 T^{1+\varepsilon}.
 \end{equation}
\end{lemma}
This off-diagonal main term gives to $I_{\Delta}^{OD}$
\begin{multline*}
M.T.^{(\Delta)}_{OD} = c' \sum_{\pm}
\sum_{1 \leq |h| \ll \frac{N A}{\Delta} T^{\varepsilon} }
 \sigma_{-1}(h) \int_{\max(0, -h)}^{\infty} \left(\frac{x + h}{x}\right)^{\mp iT}
\\
 \int_{-\infty}^{\infty} \cosh(\pi T) |\gamma_{V_T}(1/2+it)|^2 \left(\frac{t}{T}\right)^{\ell} \Psi_{T}(x+h,t) \Psi_{T}(x,t) w_{\Delta}(t) \left(\frac{x+h}{x}\right)^{it} dt
 \frac{ dx}{\sqrt{x(x+h)}},
\end{multline*}
where $c' = \frac{c}{\zeta(2)} = \frac{24}{\pi^4}$.
With a crude trivial bound, we can only show $M.T.^{(\Delta)}_{OD} = O(A T^{\varepsilon})$, which is unsatisfactory.  We shall improve on this in stages.  The first step is to learn that we may truncate the $h$-sum at a smaller point, leading to a saving by a factor $A$.

Using a first-order Taylor approximation combined with \eqref{eq:PsiDerivativeBoundsWithxAndt}, we have $\Psi_T(x+h, t) = \Psi_T(x,t) + O(N^{-1} |h| A)$.  This saves a factor $\ll \Delta^{-1} A^2 T^{\varepsilon}$ from the trivial bound.  We can similarly approximate $(x+h)^{-1/2}$, and extend the $x$-integral to the positive reals trivially from the support of $\Psi_T$.  Thus
\begin{multline}
\label{eq:MTODintermediateTerm}
M.T.^{(\Delta)}_{OD} =  c' \sum_{\pm}
\sum_{1 \leq |h| \ll \frac{N A}{\Delta} T^{\varepsilon} }
 \sigma_{-1}(h) \int_{0}^{\infty} \left(\frac{x + h}{x}\right)^{\mp iT}
\\
 \int_{-\infty}^{\infty} \cosh(\pi T) |\gamma_{V_T}(1/2+it)|^2 \left(\frac{t}{T}\right)^{\ell} \Psi_{T}(x,t)^2 w_{\Delta}(t) \left(\frac{x+h}{x}\right)^{it} dt
 \frac{ dx}{x} + O(\Delta^{-1} A^3 T^{\varepsilon}),
\end{multline}
where the error term is acceptable for Lemma \ref{lemma:MTODbound} since $T_0 \ll \Delta \ll T$.

Next we use
\begin{equation*}
 \left(\frac{x+h}{x}\right)^{it \mp iT} = e^{(it \mp iT) \log(1 + \frac{h}{x})} = e^{(it \mp iT) \frac{h}{x}} \left(1 + O\left(\frac{T h^2}{N^2}\right)\right).
\end{equation*}
This $O(\cdot)$ term is $\ll \Delta^{-2} A^2 T^{1+\varepsilon}$, and so it contributes to $M.T.^{(\Delta)}_{OD}$ an amount that is
$ \ll \Delta^{-2} A^3 T^{1+\varepsilon}$, consistent with \eqref{eq:MTODbound}.  

Taken together, these elementary estimates show
\begin{equation*}
 M.T.^{(\Delta)}_{OD} =  c' \sum_{\pm} M_{\pm} + O(\Delta^{-2} A^3 T^{1+\varepsilon}),
\end{equation*}
where
\begin{equation*}
 M_{\pm} = \sum_{h \neq 0 }
 \sigma_{-1}(h) \int_{0}^{\infty}
\\
 \int_{-\infty}^{\infty} \left(\frac{t}{T}\right)^{\ell} (T-t)^{-1/2} (T+t)^{-1/2} \psi^2 \left(\frac{\sqrt{T^2-t^2}}{2\pi x} \right) w_{\Delta}(t) e^{(it\mp iT) \frac{h}{x}} dt
 \frac{ dx}{x}.
\end{equation*}
Here we used \eqref{eq:Psidef} to replace  $\Psi_{T}(x,t)^2$ by $\psi^2 \left(\frac{\sqrt{T^2-t^2}}{2\pi x} \right)$ in \eqref{eq:MTODintermediateTerm}.

Changing variables $y = \frac{\sqrt{T^2-t^2}}{2\pi x}$,  and recalling that $w_{\Delta}(t) = w(\frac{T-t}{\Delta})$ for some fixed function $w$,
we obtain
\begin{equation*}
 M_{\pm} = \int_0^{\infty} \frac{\psi(y)^2}{y} \sum_{h \neq 0 }
 \sigma_{-1}(h)
 \int_{-\infty}^{\infty} \left(\frac{t}{T}\right)^{\ell}  (T-t)^{-1/2} (T+t)^{-1/2} w\left(\frac{T-t}{\Delta}\right)
e\left(-yh\sqrt{\frac{T\mp t}{T\pm t}}  \right)
 dt dy.
\end{equation*}

By \cite[Lemma 8.1]{BKY} (repeated integration by parts),
\begin{equation*}
 \int_{-\infty}^{\infty} \left(\frac{t}{T}\right)^{\ell}  (T-t)^{-1/2} (T+t)^{-1/2} w\left(\frac{T-t}{\Delta}\right)
e\left(-yh\sqrt{\frac{T\mp t}{T\pm t}}  \right)
 dt \ll \frac{\Delta}{(\Delta T)^{1/2}}
\left(1 + |h| \left(\frac{\Delta}{T}\right)^{\pm \frac12} \right)^{-100}.
\end{equation*}
When bounding derivatives of $(t/T)^{\ell}$,
we use that $|t| \leq T - c \Delta$ for some constant $c > 0$ fixed, which leads to the bound
\begin{equation*}
 \frac{d}{dt} \left(\frac{t}{T}\right)^{\ell} = \frac{\ell}{T} \left(\frac{t}{T}\right)^{\ell-1} \ll \frac{1}{\Delta} \frac{\ell \Delta}{T} \exp(-c\ell \frac{\Delta}{T}) \ll \Delta^{-1}.
\end{equation*}

For $M_{-}$, we have trivially that
\begin{equation*}
M_{-} \ll \| \psi^2 \|_1 \frac{\Delta}{(\Delta T)^{1/2}} \sum_{h \neq 0}  \sigma_{-1}(h)
\left(1 + |h| \left(\frac{\Delta}{T}\right)^{-\frac12} \right)^{-100} \ll \|\psi^2\|_1 \frac{\Delta}{T}.
\end{equation*}
Therefore, we obtain the following bound which is consistent with \eqref{eq:MTODbound}:
\begin{equation*}
 \sum_{\substack{1\ll \Delta \ll T \\ \Delta \text{ dyadic}}} M_{-} \ll \| \psi^2 \|_1.
\end{equation*}
Unfortunately, the same method applied to $M_{+}$ only shows $M_{+} \ll \| \psi^2 \|_1$, and summing over $\Delta$ introduces a factor $\log T$ which is the same order of magnitude as the diagonal main term.

We have so far succeeded in saving the factor $A$ and essentially removing the dependence on $\psi^2$ from the sum over $h$, and it remains to obtain a small savings in $M_{+}$, to save the factor $\log T$.  For the values of $M_{+}$ with $\Delta \geq \frac{T}{(\log T)^{100}}$, we are free to use the trivial bound $M_{+} \ll \| \psi^2 \|_1$ giving
\begin{equation*}
 \sum_{\substack{\Delta \geq \frac{T}{(\log T)^{100}} \\ \text{dyadic}}} M_{+} \ll (\log \log T) \| \psi^2 \|_1,
\end{equation*}
which is an acceptable error term.

For the rest of the proof we assume $\Delta \leq \frac{T}{(\log T)^{100}}$.  In the estimation of $M_{+}$, for the terms with $|h| \sqrt{\frac{\Delta}{T}} \geq (\log T)^{10}$, we save a power of $\log T$, and so these terms are acceptable.  Then assume
\begin{equation*}
 |h| \leq \sqrt{\frac{T}{\Delta}} (\log T)^{10}, \qquad \Delta \leq \frac{T}{(\log T)^{100}}.
\end{equation*}
In the $t$-integral in the definition of $M_{+}$, change variables $t \rightarrow T - \Delta t$, giving
\begin{multline*}
 \int_{-\infty}^{\infty} \left(\frac{t}{T}\right)^{\ell}  (T-t)^{-1/2} (T+t)^{-1/2} w\left(\frac{T-t}{\Delta}\right)
e\left(-yh\sqrt{\frac{T-t}{T + t}}  \right)
 dt
 \\
 = \frac{\Delta}{(\Delta T)^{1/2}} \int_{-\infty}^{\infty} \left(1-\frac{\Delta}{T} t\right)^{\ell} \frac{w(t)}{\sqrt{t}} e\left(\frac{-hy \Delta^{1/2} t^{1/2} }{(2T)^{1/2}} \right) (1 + O((\log T)^{-90})) dt.
\end{multline*}
This error term gives an acceptable error by trivial estimations.  We may also
observe
\begin{equation*}
 \left(1-\frac{\Delta}{T} t\right)^{\ell} = \exp\Big(\ell \Big(\frac{\Delta t}{T} + O\Big(\frac{\Delta^2}{T^2}\Big)\Big)\Big) = \exp(O((\log T)^{-99})) = 1 + O((\log T)^{-99}),
\end{equation*}
so we may discard this error term too.

Next let $t = v^2$, and define $w_2(v) = w(v^2)$.  Then we obtain
\begin{equation*}
 M_{+} = \frac{\Delta^{1/2}}{T^{1/2}} \int_0^{\infty} \frac{\psi(y)^2}{y} \sum_{0 < |h| \leq \sqrt{\frac{T}{\Delta}} (\log T)^{10}} \sigma_{-1}(h) \int_0^{\infty} w_2(v) e\left(-hy \frac{\Delta^{1/2}}{(2T)^{1/2}} v\right) dv
 + O(\|\psi^2\|_1 (\log T)^{-5}).
\end{equation*}
In the main term, we are then free to extend the sum back to all of $h \neq 0$.
Now examine the inner sum over $h$, namely
\begin{equation*}
 S_{\Delta} := \sum_{h \neq 0} \sigma_{-1}(h) \widehat{w_2}(h/V), \quad \text{where} \quad V = \frac{(2T)^{1/2}}{y \Delta^{1/2}}.
\end{equation*}
We claim that
\begin{equation*}
 \sum_{\Delta \text{ dyadic}} \frac{\Delta^{1/2}}{T^{1/2}} S_{\Delta} \ll (\log \log T),
\end{equation*}
uniformly for $y \in [1, 3]$, which suffices to complete the proof.

\begin{proof}[Proof of claim]
 Write $\sigma_{-1}(h) = \sum_{ab = h} a^{-1}$ where $a \geq 1$.  For $a > D$, with $D \geq 1$ a parameter to be chosen later, we apply a trivial bound giving
 \begin{equation*}
  \sum_{a > D} a^{-1} \sum_{b \neq 0} \widehat{w_2}(ab/V) \ll \sum_{a > D} a^{-1} \frac{V}{a} \ll \frac{V}{D}.
 \end{equation*}
For $a \leq D$, we apply Poisson summation in $b$, giving
\begin{equation*}
 \sum_{a \leq D} a^{-1} \left( -\widehat{w_2}(0) + \frac{V}{a} \sum_{l \in \mathbb{Z}} w_2\left(\frac{lV}{a}\right) \right).
\end{equation*}
Since $w_2(0) = 0$, and moreover $w_2$ has support on a fixed subset of the positive reals, if we set $D \ll V$ with a small implied constant, then the sum over $l$ is empty.  Therefore, we obtain
\begin{equation*}
 S_{\Delta} \ll \frac{(T/\Delta)^{1/2}}{D} + \log D.
\end{equation*}
We choose $D = \log T$, which satisfies the requirement $D \leq \epsilon V$ since $V \gg (\log T)^{50}$.  That is, we have shown
\begin{equation*}
\frac{\Delta^{1/2}}{T^{1/2}} S_{\Delta} \ll \frac{1}{\log T} + \frac{\Delta^{1/2}}{T^{1/2}} \log \log T.
\end{equation*}
 Summing this over dyadic values of $\Delta \ll T$ gives the claim.
\end{proof}

\section{Completion of the proof, and optimization of parameters}
Here we put together the various results in this paper and complete the proof of Theorem \ref{thm:mainthm}.

Applying Proposition \ref{prop:bulklt} to
\eqref{eq:IpsialphaPostBinomialTheorem3b}, and using $1 \ll A \| \psi^2 \|_1$ in the error term, we deduce
\begin{multline}
\label{eq:IpsialphaWithAllErrorTerms}
 I_{\psi, \alpha}(T) = \frac{6}{\pi} \log(1/4 + T^2) \sum_{\ell=0}^{L} \frac{(-i\alpha)^{\ell}}{\ell !} c_{\ell}\|\psi^2 \|_1 + O(\|\psi^2 \|_1 (\log T)^{2/3 + \varepsilon})
 \\
 + O(ET_1) + O(ET_2) + O(ET_3) + O(ET_4),
\end{multline}
where
\begin{equation*}
ET_4 =
\| \psi^2\|_1 T^{\varepsilon}\left( T^{-\frac{5}{12}} T_0^{\frac14} A^{\frac32} + T_0^{-1} A^3
+   A^4 T_0^{-\frac{9}{4}} T^{\frac{25}{12}} + A^{\frac{5}{2}} T_0^{-\frac{5}{8}} T^{\frac{13}{24}}\right).
\end{equation*}
Recall that $ET_i$ with $i=1,2,3$ are given by \eqref{eq:ET1def}, \eqref{eq:ET2def}, \eqref{eq:ET3def}.


It is easy to see that in \eqref{eq:IpsialphaWithAllErrorTerms} we may extend the sum to all $\ell \in \mathbb{N}$ without making a new error term.  The main term is evaluated with \eqref{eq:cellsumevaluation}, and gives the desired expression in Theorem \ref{thm:mainthm}.

Now we estimate all the error terms.
We have
\[
ET_1 + ET_4 \ll \|\psi^2\|_1 +  \|\psi^2 \|_1 T^{\varepsilon}\left( T^{-\frac{5}{12}} T_0^{\frac14} A^{\frac32} + T_0^{-1} A^3
+   A^4 T_0^{-\frac{9}{4}} T^{\frac{25}{12}} + A^{\frac{5}{2}} T_0^{-\frac{5}{8}} T^{\frac{13}{24}}
+A T_0^{-\frac12} T^{\frac12} + AT^{-\frac16} \right).
\]
After some experimentation, we see that this error term is minimized when
\begin{equation*}
 A T_0^{\frac12} T^{-\frac12} = A^4 T^{\frac{25}{12}} T_0^{-\frac{9}{4}},
\end{equation*}
that is, we should pick
\begin{equation*}
 T_0 = A^{\frac{12}{11}} T^{\frac{31}{33}}.
\end{equation*}
With this choice, we have
\begin{equation*}
ET_1 + ET_4 \ll \|\psi^2\|_1+ \|\psi^2\|_1 T^{\varepsilon}\left( A^{\frac{39}{22}} T^{-\frac{2}{11}} + A^{\frac{21}{11}} T^{-\frac{31}{33}} + A^{\frac{17}{11}} T^{-\frac{1}{33}} + A^{\frac{20}{11}} T^{-\frac{1}{22}} + A T^{-\frac{1}{6}}\right),
\end{equation*}
and also $T_0 \gg T^{31/33} \gg T^{9/10}$, which was a condition that is required in Propositions \ref{prop:easybound} and \ref{prop:bulklt}.

Now a direct calculation shows that if $A \ll T^{\delta}$ with $\delta < 1/51$, then $ET_1 + ET_4\ll \| \psi^2 \|_1$.

Finally, we estimate $ET_2$ and $ET_3$ with this choice of $T_0$, and range of $A$'s.  
We have $ET_2 \ll A T^{-1+\varepsilon} \|\psi^2 \|_1$, which is much more than sufficient.  For $ET_3$, we have
\begin{equation*}
 ET_3 \ll
 T^{\varepsilon}(A^{\frac{32}{11}} T^{-\frac{29}{66}} + A^{\frac{53}{11}} T^{-\frac{29}{33}} )\|\psi^2 \|_1,
\end{equation*}
which again is more than sufficient for $A \ll T^{1/51}$.

This completes the proof of Theorem \ref{thm:mainthm}.

\section{Sign changes of the Eisenstein series}
In this section, we deduce from Theorem \ref{thm:mainthm} a lower bound for the number of sign changes of the Eisenstein series on $\{iy~:~1<y<3\} \subset \mathbb{H}$, and prove Theorem \ref{thm:signchange}. To this end, we first fix $0<\delta<1/51$ and a non-negative function $\psi_0(x) \in C_0^\infty \left( -\frac{1}{2},\frac{1}{2} \right)$. Let
\[
\psi_{T,j}(y) = \psi_0 \left(\frac{y-(jT^{-\delta}+1)}{T^{-\delta}}\right)
\]
for $j=1,2,\ldots,  \lceil T^\delta \rceil$. Note that $\psi_{T,j}$ satisfies \eqref{eq:psibound}, \eqref{eq:psiintegrallowerbound}, and \eqref{eq:psiprimeandpsisquaredrelation} with $A = T^\delta$. Hence as a consequence of Theorem \ref{thm:mainthm}, there exists a sufficiently large constant $C>0$ depending only on $\delta$ and $\epsilon$, such that
\[
I_{\psi_{T,j},4} < \left(\frac{6}{\pi} \log(1/4+T^2) J_0(4) + C (\log T)^{2/3+\epsilon} \right)\|\psi_{T,j}^2\|_1
\]
for all $T$ and $j$. Observing that $J_0(4) = -0.397\ldots <0$, this in particular implies that $I_{\psi_{T,j},4} < 0$ for all sufficiently large $T$ and for each $j=1,2,\ldots, \lceil T^\delta \rceil$.

Since we have assumed that $\psi_{T,j}$ is non-negative, it follows from the definition of $I_{\psi, \alpha}$
\[
 I_{\psi, \alpha}(T) = \int_0^{\infty} \psi(y) \psi\left(\left(1+\frac{\alpha}{T}\right)y\right)
E_T^*(iy) E_T^*\left(i\left(1+\frac{\alpha}{T}\right)y\right) \frac{dy}{y},
\]
that either $E_T^*(iy)$ or $E_T^*\left(i\left(1+\frac{4}{T}\right)y\right)$ has at least one sign change on the support of $\psi_{T,j}(y) \psi_{T,j}\left(y\left(1+\frac{4}{T}\right)\right)$. Therefore $E_T^*(iy)$ has at least one sign change in the range
\[
1+\left(j-\frac{1}{2}\right) T^{-\delta}<y< 1+\left(j+\frac{1}{2}\right) T^{-\delta},
\]
for each $j=1,2,\ldots, \lceil T^\delta \rceil$. From this we conclude that $E_{T}^*(iy)$ has at least $T^{\delta}$ sign changes on the interval $1<y<3$ for all sufficiently large $T$.

\section{Acknowledgments}
The authors thank the Institute for Advanced Study for an excellent working environment when this work was initially conceived during the 2014-2015 year.

\bibliographystyle{alpha}
\bibliography{bibfile}
\end{document}